\documentclass{amsart}
\usepackage{graphicx}
\newtheorem{theorem}{Theorem}
\newtheorem{thmx}{Theorem}

\newtheorem{lemma}{Lemma}
\newtheorem{proposition}{Proposition}
\newtheorem{corollary}{Corollary}

\numberwithin{equation}{section}
\numberwithin{lemma}{section}
\numberwithin{proposition}{section}
\numberwithin{corollary}{section}
\numberwithin{remark}{section}
\begin{document}

\title{Rigidity of Schouten Tensor under Conformal Deformation}

\author{Mijia Lai}
\email{laimijia@sjtu.edu.cn}
\address{School of Mathematical Sciences, Shanghai Jiao Tong University, Shanghai 200240, China}

\author{Guoqiang Wu}
\email{gqwu@zstu.edu.cn}
\address{School of Science, Zhejiang Sci-Tech University, Hangzhou 310018, China}

\thanks{M. Lai's research is supported in part by National Natural Science Foundation of China No. 12031012, No. 12171313 and the Institute of Modern Analysis-A Frontier Research Center of Shanghai.}
\begin{abstract}
We obtain some rigidity results for metrics whose Schouten tensor is bounded from below after  conformal transformations. Liang Cheng recently proved that a complete, nonflat, locally conformally flat manifold with Ricci pinching condition ($Ric-\epsilon Rg\geq 0$) must be compact. This answers higher dimensional Hamilton's pinching conjecture on locally conformally flat manifolds affirmatively.  Since (modified) Schouten tensor being nonnegative is equivalent to a Ricci pinching condition, our main result yields a simple proof of Cheng's theorem.
\end{abstract}

\maketitle

%%%%%%
\section{Introduction}
Rigidity phenomenon in differential geometry usually concerns the uniqueness of certain geometric models under various curvature conditions. A prominent example is the rigidity part of the positive mass theorem, which asserts that if the mass of an asymptotically flat manifold is zero, then the manifold must be the standard Euclidean space. There are also many studies of rigidity phenomenon within a fixed conformal class. For example, Hang and Wang~\cite{HW} proved a hemisphere rigidity theorem. Let $g$ be a metric conformal to $g_0$, the standard round metric on a hemisphere, suppose the scalar curvature $R\geq n(n-1)$ and the induced metric on the boundary is isometric to $g_0$, then $g$ must be isometric to $g_0$. See~\cite{BCE} for a fully nonlinear generalization and ~\cite{LW} for a higher order generalization of Hang-Wang's theorem, see also ~\cite{QY}, ~\cite{BMV} from the static metric point of view.

In this note, we discover some simple rigidity principles concerning the Schouten tensor under conformal transformations. The central theme is that on complete noncompact manifolds, one cannot  increase the Schouten tensor everywhere under conformal transformation. To be more precise, let $(M,g)$ be a complete Riemannian manifold, the Schouten tensor is defined as
\[
A_g=\frac{1}{n-2}\left( Ric-\frac{R}{2(n-1)}g \right).
\]
We have the decomposition of the Riemannian tensor as
\[
Riem=\mathcal{W}_g+A_g\odot g,
\]
where $\mathcal{W}_g$ is the Weyl tensor of $g$, and $\odot$ denotes the Kulkarni-Nomizu product. Since the Weyl tensor is conformally invariant, it amounts to study the change of the Schouten tensor for conformal change of metrics. For simplicity, we shall use
\[
A_g=\left( Ric-\frac{R}{2(n-1)}g \right)
\] as the definition of the Schouten tensor in the following. For $\tau >0$, the modified Schouten tensor $A^{\tau}_g$ is defined as
\[
A^{\tau}_g=Ric-\frac{\tau R}{2(n-1)} g.
\]
Then $\tau=1$ corresponds to the Schouten tensor, while $\tau=\frac{2(n-1)}{n}$ represents the trace-less Ricci tensor.  Our main theorem is

\begin{thmx} \label{T1}
      Let $g$ be a complete metric on $\mathbb{R}^n (n\geq 3)$ which is conformal to the standard Euclidean metric $g_0$, suppose $A^{\tau}_g\geq 0$ for some $\tau\in (0, \frac{2(n-1)}{n})$, then $g$ must be isometric to $g_0$.
\end{thmx}

Extending the theme further, we have
\begin{thmx} \label{T2}
    Let $(M^n,g_0) (n\geq 3)$ be a scalar flat complete Riemannian manifold. Any metric $g$ conformal to $g_0$ with $A_g\geq A_{g_0}$ must be isometric to $g_0$.
\end{thmx}

 A famous conjecture of Hamilton asserts that a three dimensional complete manifold with {\bf pinched Ricci curvature}
 \[
 Ric_g-\epsilon Rg\geq 0 \quad \text{ for some $\epsilon>0$}
 \]
 must be either compact or flat.  Chen-Zhu~\cite{CZ} and Lott~\cite{L} confirmed this conjecture under additional curvature assumptions. Deruelle-Schulze-Simon~\cite{DSS} removed the inverse quadratic lower bound assumption in Lott's work. Most recently Lee and Topping~\cite{LT} resolved Hamilton's conjecture completely via Ricci flow, another proof is given by Huisken-Koerber~\cite{HK} using inverse mean curvature flow. Xu~\cite{X} obtained an interesting formula for the integral of scalar curvature in terms of the asympotic volume ratio on three dimensional complete manifolds with a pole and nonnegative Ricci curvature. As a consequence, he gave a simple proof of Hamilton's conjecture in that case. In higher dimension, Cheng~\cite{C} considered the case when $(M,g)$ is locally conformally flat. Using the Yamabe flow, he is able to show that a locally conformal flat manifold with $Ric-\epsilon Rg\geq 0$ for some $\epsilon>0$ must be either compact or flat.

Note the condition $A^{\tau}_g\geq 0$ is equivalent to Ricci pinching condition, combining with the classification of locally conformally flat manifolds with nonnegative Ricci curvature, our main theorem yields a simple proof of Cheng's theorem.

\begin{corollary}
Let $\left(M^n, g\right) (n \geq 3)$ be an $n$-dimensional non-flat complete locally conformally flat Riemannian manifold satisfying
$$
Ric-\epsilon R g \geq 0.
$$
for some $\epsilon\in(0,\frac{1}{n})$, then $M^n$ must be compact.
\end{corollary}
\begin{proof}
Thanks to the classification of complete locally conformally flat manifolds with nonnegative Ricci curvature by Zhu~\cite{Z} and Carron-Herzlich~\cite{CH}, $(M,g)$ is one of the following:
\begin{enumerate}
    \item $M$ is non-flat and globally conformally equivalent to $\mathbb{R}^n$;
    \item $M$ is globally conformally equivalent to a space form of positive curvature;
    \item $M$ is locally isometric to the cylinder $\mathbb{R}\times S^{n-1}$;
    \item $M$ is isometric to a complete flat manifold.
\end{enumerate}
As (3) violates the pinching condition, to prove the compactness of the manifold, it suffices to rule out case (1), which follows directly from Theorem~\ref{T1}.
\end{proof}

\section{Proofs of main theorems}
The main idea behind the proof is to explore the effect of the positivity of the (modified) Schouten tensor on the conformal factor. It turns out the positivity in radial directions is strong enough to give a proper growth control for the conformal factor, which leads to a contradiction to the completeness of the metric.

Now we prove Theorem~\ref{T1}. Indeed, we prove a stronger version.
\begin{theorem}
    Let $g$ be a complete metric on $\mathbb{R}^n$ ($n\geq 3$) which is conformal to the Euclidean metric $g_0$ with nonnegative scalar curvature $R_g\geq0$. Suppose there exists $p\in \mathbb{R}^n$, and $\tau\in (0, \frac{2(n-1)}{n})$ such that
    \[
    \int_{S_r(p)} A^{\tau}_g(\partial_r, \partial_r) d\sigma \geq 0 \quad \forall r>0,
    \] where $S_r(p)$, $\partial_r$, $d\sigma$ are all with respect to $g_0$.  Then $g$ must be isometric to $g_0$.
\end{theorem}

\begin{proof} For simplicity, we may assume $p=0$.
Assume $g=\phi^{-2} g_0$, then
\begin{align} \notag
    Ric_{g}=(n-2)\frac{\nabla^2 \phi}{\phi}+(\frac{\Delta \phi}{\phi}-(n-1)\frac{|\nabla \phi|^2}{\phi^2})g_0,
\end{align}
and
\begin{align} \label{scalar}
    R_g=\phi^2 (2(n-1)\frac{\Delta \phi}{\phi}-(n-1)n \frac{|\nabla\phi|^2}{\phi^2}).
\end{align}

Thus
\[
A^{\tau}_g=(n-2)\frac{\nabla^2 \phi}{\phi}+[(1-\tau)\frac{\Delta \phi}{\phi}-(n-1-\frac{\tau n}{2})\frac{|\nabla \phi|^2}{\phi^2}] g_0.
\]
Evaluating at $\partial_r$ and multiplying both sides by $\phi$, we get
\begin{align} \label{1}
    \phi A^{\tau}(\partial_r, \partial_r)=(n-2) \frac{\partial^2 \phi}{\partial r^2}+(1-\tau)(\frac{\partial^2 \phi}{\partial r^2}+\frac{n-1}{r}\frac{\partial \phi}{\partial r}+\bar{\Delta} \phi)-(n-1-\frac{\tau n}{2})\frac{|\nabla \phi|^2}{\phi},
    \end{align}
where $\bar{\Delta}$ denotes the Laplacian on $S^{n-1}(r)$.

Let $\varphi(r)=\frac{1}{|S^{n-1}(r)|}\int_{S^{n-1}(r)} \phi (x) d\sigma$ be the spherical average of $\phi$. Taking the spherical average of both sides of (\ref{1}) and using $A^{\tau}\geq 0$, we thus have
\begin{align} \label{2}
(n-1-\tau)\varphi''(r)+\frac{(1-\tau)(n-1)}{r} \varphi'(r)-(n-1-\frac{\tau n}{2}) \frac{1}{|S^{n-1}(r)|}\int_{S^{n-1}(r)} \frac{|\nabla \phi|^2}{\phi} d\sigma \geq 0.
\end{align}
Using $|\nabla \phi|^2\geq (\frac{\partial \phi}{\partial r})^2$ and H\"{o}lder's inequality
\[
\int_{S^{n-1}(r)} \frac{(\frac{\partial \phi}{\partial r})^2}{\phi} \int_{S^{n-1}{(r)}} \phi \geq \left( \int_{S^{n-1}(r)} \frac{\partial \phi}{\partial r}
\right)^2,
\]
we proceed (\ref{2}) to get
\begin{align} \label{3}
    (n-1-\tau)\varphi''(r)+\frac{(1-\tau)(n-1)}{r} \varphi'(r)-(n-1-\frac{\tau n}{2})\frac{(\varphi'(r))^2}{\varphi(r)}\geq 0.
\end{align}
Let $\psi(t)=\varphi(e^t)$, then (\ref{3}) simplifies to
\begin{align} \label{4}
     (n-1-\tau)\psi''(t)-\tau(n-2)\psi'(t) -(n-1-\frac{\tau n}{2})\frac{(\psi'(r))^2}{\psi(r)}\geq 0.
\end{align}

Let $\alpha=\frac{\tau(n-2)}{n-1-\tau}$ and $\beta=\frac{n-1-\frac{\tau n}{2}}{n-1-\tau}$, (\ref{4}) can be written as
\[
\psi''-\alpha \psi'-\beta \frac{\psi'^2}{\psi}\geq0.
\]
Multiplying both sides by $(1-\beta)\psi^{-\beta}$, ($\beta<1$ for $\tau\in (0, \frac{2(n-1)}{n})$) we get
\[
(\psi^{1-\beta})''-\alpha (\psi^{1-\beta})'\geq 0.
\]
Note $\psi(-\infty)$ exists and $\psi'(-\infty)=0$, we infer that  $\psi$ is either a constant or
\[
\psi^{1-\beta}(t)\geq  C_1e^{\alpha t}+C_2.
\]

We shall rule out the later case. Note $\frac{\alpha}{1-\beta}=2$, so in the latter case, there exists a uniform $C$ such that
\[
\varphi(r)\geq C r^2, \quad \text{for $r$ large}.
\]
In view of (\ref{scalar}) and $R_g\geq0 $, we have that $\phi$ is a positive subharmonic function, which satisfies
\begin{align} \label{5}
\frac{1}{|S^{n-1}(r)|}\int_{S^{n-1}(r)} \phi (x) d\sigma\geq Cr^2.
\end{align}

At this point, we recall a fine asymptotic behavior of the conformal factor due to Ma-Qing (Theorem 5.2 in ~\cite{MQ}). Let $g=e^{2f}g_0$ be a complete metric with nonnegative Ricci curvature on $\mathbb{R}^n$, then there exists $m\in [0,1]$ and universal constant $C$ such that
\[
f(x)\geq -m \ln |x|-C.
\]
Therefore
\[
\phi(x)\leq e^{-f}\leq \frac{C}{|x|^m} \quad \text{for $|x|$ large}.
\]
This clearly contradicts to (\ref{5}) as $m\in [0,1]$, and the proof is complete.

We also offer an elementary proof which does not rely on Ma-Qing's estimate in the case $n\geq 4$. The idea of this proof appears in the proof of Theorem~\ref{T2} as well.

Since $n\geq 4$, an easy computation shows that $u=\phi^{-1}$ is a positive superharmonic function on $\mathbb{R}^n$. By mean value property, there exists a uniform constant $C$ such that
\[
\int_{B_{2R}} u(x) dx\leq C R^n \inf_{B_R} u.
\]
It follows from (\ref{5}) that $\inf_{B_R} u \leq \frac{C}{R^2}$.
Thus
\[
\int_{B_{2R}\setminus B_R} \frac{u(x)}{|x|^{n-1}} dx\leq \frac{1}{R^{n-1}} \int_{B_{2R}\setminus B_R} u(x) dx\leq \frac{1}{R^{n-1}} \int_{B_{2R}} u(x) dx \leq \frac{C}{R}.
\]
Taking $R=2^k$ for $k=0,1,2,\cdots$, we get that
\[
\int_{\mathbb{R}^n\setminus B_1} \frac{u(x)}{|x|^{n-1}} dx<\infty.\]

Under polar coordinates, we have
\[
\int_{\mathbb{R}^n\setminus B_1} \frac{u(x)}{|x|^{n-1}} dx =\int_{S^{n-1}(1)} \int_{1}^{\infty} u(r \theta) dr d\sigma,
\]
Thus
\[
\int_{1}^{\infty} u(r \theta) dr<\infty, \text{a.e. $\theta\in S^{n-1}(1)$}.
\]
This means the $g$-length of a divergent path is finite, contradicting to the fact that $g$ is complete.

Hence $\psi$ has to be a constant function, tracing the case when equalities hold in the above argument shows that $\phi$ is also a constant. Therefore $g$ is isometric to $g_0$.
\end{proof}

The proof of Theorem ~\ref{T2} is a reminiscent of the above proof for $n\geq 4$.  However, we just need to explore the growth control of the conformal factor along a particular choosen geodesic ray.

\begin{proof}[Proof of Theorem~\ref{T2}]
Suppose $g=e^{2f}g_0$, then the corresponding Ricci tensor and scalar curvature are
\[
Ric_g=Ric_0-(n-2)(\nabla^2 f-df \otimes df)-(\Delta f+(n-2)|\nabla f|^2) g_0,
\]
\[
R_g=R_0e^{-2f}-2(n-1)e^{-2f} \Delta f-(n-1)(n-2)e^{-2f}|\nabla f|^2.
\]
Hence the condition $A_g\geq A_{g_0}$ is equivalent to
\begin{align} \notag
-(n-2)(\nabla^2 f-df \otimes df)-\frac{n-2}{2} |\nabla f|^2g_0\geq 0.
\end{align}

Restricting above inequality to a geodesic ray $\gamma$ of $(M, g_0)$, we find that
\[
-f''(\gamma(s))+ (f'(\gamma(s)))^2-\frac{1}{2} |\nabla f|^2 \geq 0.
\]
Since $|\nabla f|^2\geq f'(\gamma(s))^2$, it follows that
\begin{align} \notag
f''(\gamma(s))-\frac{1}{2}(f'(\gamma(s)))^2\leq 0.
\end{align}

Since $R_0=0$ and $A_g\geq A_{g_0}$, if follows that $R_g\geq 0$. Consequently, from the equation of conformal change of the scalar curvature, we infer that $f$ is a superharmonic function on $(M, g_0)$. If $f$ is not constant, in view of the maximum principle, there exists a geodesic ray, such that $f'(\gamma(s_0))<0$ for some $s_0$. We claim that $f'(\gamma(s))\leq 0$ for $s\geq s_0$. Indeed, since
\[
(f'e^{-\frac{1}{2}f})'=e^{-\frac{1}{2}f}[f''-\frac{1}{2}f']\leq 0,
\]
thus if $f'(\gamma(s_1))=0$, then $f'(\gamma(s))\leq 0$ $s\geq s_1$. Therefore if $f'(\gamma(s_0))<0$, $f'(\gamma(s))$ stays nonpositive.

For simplicity, we denote $f'(\gamma(s))$ by $h(s)$. Thus  $h(s)\leq 0$ satisfies
\[
h'(s)-\frac{1}{2}h^2(s)\leq 0, \quad \text{for $s\geq s_0$}.
\]
It is easy to see that $h(s)\leq -\frac{2}{s+C}$ for some constant $C$. Consequently, $f(\gamma(s))\leq -2 \ln(s+C)$. This implies that the $g$-length of the divergent curve $\gamma$ is finite, contradicting to the completeness of $g$.
\end{proof}

We conclude this paper with following discussion. A metric $g$ is called $\Gamma_k$-positive (nonnegative) if $\sigma_j(A_g)(x)>(\geq)0$ for all $j\leq k$ and $x\in M$. We denote it by $A_g\in \Gamma_k(\overline{\Gamma_k})$.  Guan-Viaclovsky-Wang~\cite{GVW} have discovered an interesting property of the Schouten tensor, which states that a $\Gamma_k$-positive metric for some $k\geq \frac{n}{2}$ has positive Ricci curvature. Following their proof, a similar conclusion holds for the modified Schouten tensor.
\begin{proposition} Let $\tau\in (0, \frac{2(n-1)}{n})$ be fixed, then
    \[
 A^{\tau}_{g}\in \overline{\Gamma_{k}}\Longrightarrow Ric_g\geq \frac{(2k-(2-\tau)n)(n-1)}{2k(n-1)-n\tau}R g.
\]
\end{proposition}

For the convenience of the reader, we provide a detailed proof. The proposition will follow directly from the following algebraic lemma.

\begin{lemma}
    Let $\Lambda=(\lambda_1, \cdots, \lambda_n)\in\mathbb{R}^n$. For a fixed $\mu\in(0, \frac{1}{n})$, set
    \[S_{\Lambda}=\Lambda-\mu \sum_{i=1}^{n} \lambda_i (1, 1, \cdots, 1).
    \]Suppose $S_{\Lambda}\in \overline{\Gamma_{k}}$, then
    \[
    \min_{i} \lambda_i \geq \frac{(2k-(2-\tau)n)(n-1)}{2k(n-1)-n\tau} \sum_{i=1}^n \lambda_i.
\]
\end{lemma}

\begin{proof}
    Let $\Lambda_0=(1, \cdots, 1, \delta_k)$ such that $S_{\Lambda_0}\in \overline{\Gamma_{k}}$.
Write $S_{\Lambda_0}=(a, \cdots, a, b)$, with
\[
a=1-\mu(n-1+\delta_k), \quad b=\delta_k-\mu(n-1+\delta_k).
\]
It is easy to compute that for $\delta_k=\frac{\mu(n-1)n-(n-k)}{k-n\mu}$, $S_{\Lambda_0}\in \partial \Gamma_{k}^{+}$.  Since $\mu<\frac{1}{n}$, we have that $\delta_k<1$ and $b<a$.

Now take any $\Lambda$ such that $S_{\Lambda}\in \overline{\Gamma_{k}}$, we may assume that $\sum_{i=1}^n \lambda_i=n-1+\delta_k$ and $\lambda_n=\min_{i} \lambda_i$. Write $S_{\Lambda}=(a_1, a_2, \cdots, a_n)$, to prove the lemma, it suffices to show that $a_n\geq b$.

Set $S_t=(1-t)S_{\Lambda_0}+t S_{\Lambda}$ for $t\in[0,1]$. In view of the convexity of $\overline{\Gamma_k}$, $f(t)=\sigma_k(S_t)\geq 0$, $\forall t\in [0,1]$ and $f(0)=0$. Thus
\[
0\leq f'(0)=\sum_{i=1}^{n-1} \sigma_{k-1}(S_{\Lambda_0} | i) (a_i- a) +\sigma_{k-1}(S_{\Lambda_0}| n)(a_n-b).
\]
Note $\sum_{i=1}^{n} a_i=(n-1)a+b$ and $\sigma_{k-1}(S_{\Lambda_0} | i)$ are all equal for $i=1, 2, \cdots, n-1$, thus above inequality implies
\[
(a_n-b) \left( \sigma_{k-1}(S_{\Lambda_0}| n)-\sigma_{k-1}(S_{\Lambda_0}| 1)\right)\geq 0.
\]
Since $a>b$, it follows that  $\sigma_{k-1}(S_{\Lambda_0}| n)\geq\sigma_{k-1}(S_{\Lambda_0}| 1)$. Thus $a_n\geq b$.

\end{proof}

Therefore, our main theorem implies that a complete, nonflat, locally conformally flat manifold with $A_g^{\tau}\in \overline{\Gamma_k}$ for $k>\frac{(2-\tau)n}{2}$ must be compact. Guan-Viaclovsky-Wang~\cite{GVW} proved that if $(M,g)$ is a compact, locally conformally flat and $g$ is $\Gamma_k$-positive for some $k>\frac{n}{2}$, then $(M,g)$ must be conformal to a spherical space form. Now the word 'compact' can be deleted in the above statement.

\end{document}